\documentclass{amsart}
\usepackage{amssymb}
\usepackage{amsmath}
\usepackage{amsfonts}

\newtheorem{theorem}{Theorem}
\theoremstyle{plain}

\newtheorem{corollary}{Corollary}

\numberwithin{equation}{section}
\input{tcilatex}

\begin{document}
\title[Functions With Fourier Transforms in Weighted Amalgam Spaces]{The
Banach Algebra of Functions With Fourier Transforms in Weighted Amalgam
Spaces}
\author{Cihan UNAL}
\address{Sinop University\\
Faculty of Arts and Sciences\\
Department of Mathematics}
\email{cihanunal88@gmail.com}
\urladdr{}
\thanks{}
\author{Ismail AYDIN}
\address{Sinop University\\
Faculty of Arts and Sciences\\
Department of Mathematics}
\email{iaydin@sinop.edu.tr}
\urladdr{}
\thanks{}
\subjclass[2000]{Primary 43A15; Secondary 46E30, 43A22}
\keywords{Amalgam spaces, inclusion, compact embedding, multipliers, Fourier
transform}
\dedicatory{}
\thanks{}

\begin{abstract}
In this paper, we define $A_{\vartheta _{1},\vartheta _{2}}^{p,1,q,r}\left(
G\right) $ to be space of all functions in $\left( L_{\vartheta
_{1}}^{p},\ell ^{1}\right) $ whose Fourier transforms belong to $\left(
L_{\vartheta _{2}}^{q},\ell ^{r}\right) .$ Moreover, we consider the basic
and advance properties of this space including Banach algebra, translation
invariant, Banach module, a generalized type of Segal algebra etc. Also, we
study some inclusions, compact embeddings in sense to weights and further
discuss multipliers of this space.
\end{abstract}

\maketitle

\section{Introduction}

Let $G$ be a locally compact abelian group with Haar measure $\mu $. For $%
1\leq p,q\leq \infty ,$ an amalgam space $\left( L^{p},\ell ^{q}\right)
\left( G\right) $ is to be a Banach space of all measurable functions on $G$
which belong locally to $L^{p}$ and globally to $\ell ^{q}$. Many authors
have considered special cases of amalgam spaces such as Krogstad \cite{Krog}%
, Liu et al. \cite{Liu2}, Szeptycki \cite{Sze}, Wiener \cite{Wie1}. Also,
Holland \cite{Holl} presented an important study for amalgams on the real
line. In 1979, Stewart \cite{Ste} extended Holland's definition to locally
compact abelian groups for locally compact groups by the Structure Theorem.

For $1\leq p<\infty ,$ the set $A_{p}\left( 
\mathbb{R}
^{d}\right) $ is the space of all complex-valued functions in $L^{1}\left( 
\mathbb{R}
^{d}\right) $ whose the Fourier transforms belong to $L^{p}\left( 
\mathbb{R}
^{d}\right) .$ This space have considered some authors, see \cite{Lai}, \cite%
{Lar}, \cite{M}. Moreover, Feichtinger and Gurkanli \cite{Feich3}, Fischer
et al. \cite{Fis} and Gurkanli \cite{Gur3} have found some generalized
results in sense to weights.

In 1926, Wiener \cite{Wie1} presented the first systematical work on amalgam
spaces $\left( L^{p},\ell ^{q}\right) $. The usage areas of the amalgam
spaces are generally harmonic and time-frequency analysis. For an another
historical journey, we can refer \cite{Four}. Moreover, the amalgam spaces
or some special cases of these spaces were investigate by a number of
authors, see \cite{Bert}, \cite{Bus}, \cite{Gur2}, \cite{Heil1}.

In this paper, we define a space $A_{\vartheta _{1},\vartheta
_{2}}^{p,1,q,r}\left( G\right) =\left\{ f\in \left( L_{\vartheta
_{1}}^{p},\ell ^{1}\right) :\widehat{f}\in \left( L_{\vartheta
_{2}}^{q},\ell ^{r}\right) \right\} $ and investigate the basic properties
of the space. Also, we consider several inclusions under some conditions.
Moreover, we discuss compact embeddings with other suitable spaces under
some conditions and reveal multipliers of the space $A_{\vartheta
_{1},\vartheta _{2}}^{p,1,q,r}\left( G\right) .$ One of the our purpose of
this paper is to generalize some of the results in \cite{A2} and \cite{Unal}
to the double weighted case.

\section{Notation and Preliminaries}

Throughout this paper, we will work on $G$ with Lebesgue measure $dx$. We
denote by $C_{c}\left( G\right) $ as the linear space of continuous
functions on $G$, which have compact support. The translation and\ character
operators $T_{y}$ and $M_{t}$ are defined by $T_{y}f(x)=f(x-y)$ and $%
M_{t}f(y)=\left\langle y,t\right\rangle .f(y)$ respectively for $x,y\in G$ , 
$t\in G.$ Also $(B,\left\Vert .\right\Vert _{B})$ is strongly translation
invariant if one has $T_{y}B\subseteq B$ and $\left\Vert T_{y}f\right\Vert
_{B}=\left\Vert f\right\Vert _{B}$ and strongly character invariant if $%
M_{t}B\subseteq B$ and $\left\Vert M_{t}f\right\Vert _{B}=\left\Vert
f\right\Vert _{B}$ for all $f\in B$, $y\in G$ and $t\in G$.

A measurable and locally integrable function $\vartheta :G\longrightarrow
\left( 0,\infty \right) $ is called a weight function. Moreover the weight $%
\vartheta $ will be called a Beurling's weight function if $\vartheta \left(
x\right) \geq 1$\ and $\vartheta (x+y)\leq \vartheta \left( x\right)
\vartheta \left( y\right) $ for all $x,y\in G$. We say that $\vartheta
_{1}\prec \vartheta _{2}$ if and only if there exists a $C>0$ such that $%
\vartheta _{1}\left( x\right) \leq C\vartheta _{2}\left( x\right) $ for all $%
x\in G$. Also, if $\vartheta _{1}\prec \vartheta _{2}$ and $\vartheta
_{2}\prec \vartheta _{1}$ are satisfied, then we say that weight functions $%
\vartheta _{1}$ and $\vartheta _{2}$ are equivalent and denoted by $%
\vartheta _{1}\thickapprox \vartheta _{2}$. A weight function $\vartheta $
is said to satisfy the Beurling-Domar (shortly BD) condition if 
\begin{equation*}
\sum\limits_{n\geq 1}\frac{\log \vartheta \left( nx\right) }{n^{2}}<\infty
\end{equation*}%
for all $x\in G,$ see \cite{Domar}. Moreover, it is clear that every weight
function is equivalent to a continuous weight, see \cite[Lemma 4]{Mur}.
Hence, we deduce that $\vartheta \left( x\right) \longrightarrow \infty $ as 
$x\longrightarrow \infty .$ For example, if we choose the polynomial type
weight function $\vartheta _{s}\left( x\right) =\left( 1+\left\vert
x\right\vert \right) ^{s}$ for $s\geq 0,$ then we have $\vartheta _{s}\left(
x\right) \longrightarrow \infty $ as $x\longrightarrow \infty ,$ see \cite%
{Reit}.

For $1\leq p<\infty ,$ the weighted Lebesgue space $L_{\vartheta }^{p}\left(
G\right) =\left\{ f:f\vartheta \in L^{p}\left( G\right) \right\} $ is a
Banach space with norm $\left\Vert f\right\Vert _{p,\vartheta }=\left\Vert
f\vartheta \right\Vert _{p}$ and its dual space $L_{\vartheta
^{-1}}^{q}\left( G\right) ,$ where $\frac{1}{p}+\frac{1}{q}=1.$ It is known
that the space $L_{\vartheta }^{p}\left( G\right) $ is a reflexive Banach
space for $1<p<\infty $. Moreover, for $p=1,$ $L_{\vartheta }^{1}\left(
G\right) $ is a Banach algebra under convolution, called a Beurling algebra.
It is obvious that $\left\Vert .\right\Vert _{1}\leq \left\Vert .\right\Vert
_{1,\vartheta }$ and $L_{\vartheta }^{1}\left( G\right) \subset L^{1}\left(
G\right) $. It is known that $L_{\vartheta _{2}}^{p}\left( G\right) \subset
L_{\vartheta _{1}}^{p}\left( G\right) $ if and only if $\vartheta _{1}\prec
\vartheta _{2}$, see \cite{Feich3}.

Now, assume that $A$ is a Banach algebra. It is known that a Banach space $B$
is called a Banach $A$-module if there exists a bilinear operation $\cdot
:A\times B\longrightarrow B$ such that

\begin{enumerate}
\item[\textit{(i)}] $\left( f\cdot g\right) \cdot h=f\cdot \left( g\cdot
h\right) $ for all $f,g\in A$, $h\in B.$

\item[\textit{(ii)}] For some constant $C\geq 1,$ $\left\Vert f\cdot
h\right\Vert _{B}\leq C\left\Vert f\right\Vert _{A}\left\Vert h\right\Vert
_{B}$ for all $f\in A,$ $h\in B,$ see \cite{Doran}.
\end{enumerate}

Moreover, $L_{loc}^{p}\left( G\right) $ is the space of all functions on $G$
such that $f$ restricted to any compact subset $K$ of $G$ belongs to $%
L^{p}\left( G\right) $. For $p=1,$ the space $L_{loc}^{1}\left( G\right) $
is to be space of all measurable functions $f$ on $G$ such that $f.\chi
_{K}\in L^{1}\left( G\right) $ for any compact subset $K\subset G$.
Moreover, a Banach function space (shortly BF-space) on $G$ is a Banach
space $\left( B,\left\Vert .\right\Vert _{B}\right) $ of measurable
functions which is continuously embedded into $L_{loc}^{1}\left( G\right) $,
i.e. for any compact subset $K\subset G$ there exists some constant $C_{K}>0$
such that $\left\Vert f.\chi _{K}\right\Vert _{L^{1}}\leq C_{K}.\left\Vert
f\right\Vert _{B}$ for all $f\in B$. Also, a BF-space is called solid if $%
g\in B,$ $f\in L_{loc}^{1}\left( G\right) $ and $\left\vert f\left( x\right)
\right\vert \leq \left\vert g\left( x\right) \right\vert $ locally almost
everywhere (shortly l.a.e) implies $f\in B$ and $\left\Vert f\right\Vert
_{B}\leq \left\Vert g\right\Vert _{B}$. It is easy to see that $\left(
B,\left\Vert .\right\Vert _{B}\right) $ is solid if and only if it is a $%
L^{\infty }$-module. Let $f$\ be a measurable function on $G.$

Suppose that $V$ and $W$ are two Banach modules over a Banach algebra $A$.
Then a multiplier from $V$ into $W$ is a bounded linear operator $%
T:V\longrightarrow W$, which commutes with module multiplication, i.e. $%
T(av)=aT(v)$ for $a\in A$ and $v\in V$. Also, we denote by $Hom_{A}\left(
V,W\right) $ as the space of all multipliers from $V$ into $W.$ For
convenience, we write that $Hom_{A}\left( V,V\right) =Hom_{A}\left( V\right) 
$. It is known that 
\begin{equation*}
Hom_{A}\left( V,W^{\ast }\right) \cong \left( V\otimes _{A}W\right) ^{\ast }
\end{equation*}%
where $W^{\ast }$ is dual of $W$ and $V\otimes _{A}W$ is the $A$-module
tensor product of $V$ and $W$,see \cite[Corollary 2.13]{Rief}.

Moreover, the space $M\left( G\right) $ denotes all bounded regular Borel
measures on $G.$ Now, we define%
\begin{equation*}
M\left( \vartheta \right) =\left\{ \mu \in M\left( G\right)
:\int\limits_{G}\vartheta d\left\vert \mu \right\vert <\infty \right\} .
\end{equation*}%
It is known that the space of multipliers from $L_{\vartheta }^{1}\left(
G\right) $ to $L_{\vartheta }^{1}\left( G\right) $ is homeomorphic to the
space $M\left( \vartheta \right) ,$see \cite{Gau}.

In \cite{Cigler}, Cigler revealed a generalization of Segal algebra. To
define this we suppose that $S_{\vartheta }(G)=S_{\vartheta }$ is a
subalgebra of $L_{\vartheta }^{1}(G)$ satisfying the conditions below.

\begin{enumerate}
\item[(S1)] $S_{\vartheta }$ is dense in $L_{\vartheta }^{1}(G).$

\item[(S2)] $S_{\vartheta }$ is a Banach algebra under some norm $\left\Vert
.\right\Vert _{S_{\vartheta }}$ and invariant under translations.

\item[(S3)] $\left\Vert T_{y}f\right\Vert _{S_{\vartheta }}\leq \vartheta
(y)\left\Vert f\right\Vert _{S_{\vartheta }}$ for all $y\in G$ and for each $%
f\in S_{\vartheta }$.

\item[(S4)] If $f\in S_{\vartheta }$, then for every $\varepsilon >0$ there
is a neighborhood $U$ of the identity element of $G$ such that $\left\Vert
T_{y}f-f\right\Vert _{S_{\vartheta }}<\varepsilon $ for all $y\in U.$

\item[(S5)] $\left\Vert f\right\Vert _{1,\vartheta }\leq \left\Vert
f\right\Vert _{S_{\vartheta }}$ for all $f\in S_{\vartheta }.$
\end{enumerate}

Denote the amalgam of $L^{p}$ and $\ell ^{q}$ on the real line is the normed
space 
\begin{equation*}
\left( L^{p},\ell ^{q}\right) =\left\{ f\in L_{loc}^{p}\left( 
\mathbb{R}
\right) :\left\Vert f\right\Vert _{pq}<\infty \right\}
\end{equation*}%
equipped with the norm%
\begin{equation}
\left\Vert f\right\Vert _{pq}=\left[ \tsum\limits_{n=-\infty }^{\infty }%
\left[ \tint\limits_{n}^{n+1}\left\vert f(x)\right\vert ^{p}dx\right] ^{%
\frac{q}{p}}\right] ^{\frac{1}{q}}.  \label{Normpq}
\end{equation}%
We make the appropriate changes for $p$, $q$ infinite. The norm (\ref{Normpq}%
) makes the amalgam space $\left( L^{p},\ell ^{q}\right) $ into a Banach
space, see \cite{Holl}. Note that, the space $C_{c}(G)$ is a subspace of
every amalgam spaces. Now, let $1\leq p,q<\infty $. Then, the dual space of $%
\left( L^{p},\ell ^{q}\right) $ is isometrically isomorphic to $\left(
L^{p^{\shortmid }},\ell ^{q^{\shortmid }}\right) $ where $\frac{1}{p}+\frac{1%
}{p^{\prime }}=\frac{1}{q}+\frac{1}{q^{\prime }}=1$, see \cite{A2}, \cite%
{Squ}.

Stewart \cite{Ste} give an alternative definition of $\left( L^{p},\ell
^{q}\right) \left( G\right) $ based on the Structure Theorem \cite[Theorem
24.30]{Hew}. Indeed, let $G=%
\mathbb{R}
^{a}\times G_{1}$, where $a$ is a nonnegative integer and $G_{1}$ is a
locally compact abelian group which contains an open compact subgroup $H$.
Also, we denote $I=\left[ 0,1\right) ^{a}\times H$ and $J=%
\mathbb{Z}
^{a}\times T$ where $T$ is a transversal of $H$ in $G_{1}$, i.e. $%
G_{1}=\tbigcup\limits_{t\in T}\left( t+H\right) $ is a coset decomposition
of $G_{1}$. For $\alpha \in J,$ we define $I_{\alpha }=\alpha +I.$ Therefore 
$G$ equals the disjoint union of relatively compact sets $I_{\alpha }$. We
normalize $\mu $ such that $\mu \left( I\right) =\mu \left( I_{\alpha
}\right) =1$ for all $\alpha $. Let $1\leq p,q\leq \infty $. The amalgam
space $\left( L^{p},\ell ^{q}\right) \left( G\right) =\left( L^{p},\ell
^{q}\right) $ is a Banach space%
\begin{equation*}
\left\{ f\in L_{loc}^{p}\left( G\right) :\left\Vert f\right\Vert
_{pq}<\infty \right\} ,
\end{equation*}%
where%
\begin{eqnarray}
\left\Vert f\right\Vert _{pq} &=&\left[ \tsum\limits_{\alpha \in
J}\left\Vert f\right\Vert _{L^{p}(I_{\alpha })}^{q}\right] ^{1/q}\text{\ if }%
1\leq p,q<\infty ,  \label{Normpq2} \\
\left\Vert f\right\Vert _{\infty q} &=&\left[ \tsum\limits_{\alpha \in J}%
\underset{x\in I_{\alpha }}{\sup }\left\vert f(x)\right\vert ^{q}\right]
^{1/q}\text{ if }p=\infty ,\text{ }1\leq q<\infty ,  \notag \\
\left\Vert f\right\Vert _{p\infty } &=&\underset{\alpha \in J}{\sup }%
\left\Vert f\right\Vert _{L^{p}(I_{\alpha })}\text{ if }1\leq p<\infty \text{%
, }q=\infty .  \notag
\end{eqnarray}%
If $G=%
\mathbb{R}
$, then we have $J=%
\mathbb{Z}
$, $I_{\alpha }=\left[ \alpha ,\alpha +1\right) $ and (\ref{Normpq2})
becomes (\ref{Normpq}).

Throughout this paper, $G$ will denote a locally compact abelian group with
Haar measure and $J$ and $I_{\alpha }$ will define as above. Moreover, we
assume that $1\leq p,q,r<\infty $ and every weights we used are Beurling's
weight functions on $G$.

\section{\textbf{The Space }$A_{\protect\vartheta _{1},\protect\vartheta %
_{2}}^{p,1,q,r}\left( G\right) $}

Now, we define the vector space $A_{\vartheta _{1},\vartheta
_{2}}^{p,1,q,r}\left( G\right) =\left\{ f\in \left( L_{\vartheta
_{1}}^{p},\ell ^{1}\right) :\widehat{f}\in \left( L_{\vartheta
_{2}}^{q},\ell ^{r}\right) \right\} $ and equip with the norm%
\begin{equation*}
\left\Vert f\right\Vert _{\vartheta _{1},\vartheta
_{2}}^{p,1,q,r}=\left\Vert f\right\Vert _{p1,\vartheta _{1}}+\left\Vert 
\widehat{f}\right\Vert _{qr,\vartheta _{2}}
\end{equation*}%
for $f\in A_{\vartheta _{1},\vartheta _{2}}^{p,1,q,r}\left( G\right) $. It
is note that, since $\left( L_{\vartheta _{1}}^{p},\ell ^{1}\right) $ is a
subspace of $L^{1}\left( G\right) ,$ the Fourier transforms of the functions
in $\left( L_{\vartheta _{1}}^{p},\ell ^{1}\right) $ are well-defined.

The proof of the following theorem is clear, see \cite{Ayd1}.

\begin{theorem}
\label{obvious}Let $1\leq p,q<\infty $. Let $\left( f_{n}\right) _{n\in 
\mathbb{N}
}$ be a sequence in $\left( L_{\vartheta _{1}}^{p},\ell ^{q}\right) $ and $%
\left\Vert f_{n}-f\right\Vert _{pq,\vartheta _{1}}\longrightarrow 0$ as $%
n\longrightarrow \infty $, where $f\in \left( L_{\vartheta _{1}}^{p},\ell
^{q}\right) $. Then $\left( f_{n}\right) _{n\in 
\mathbb{N}
}$ has a subsequence which converges pointwise almost everywhere to $f$.
\end{theorem}

\begin{theorem}
\label{banach}The space $A_{\vartheta _{1},\vartheta _{2}}^{p,1,q,r}\left(
G\right) $ is a Banach space with respect to $\left\Vert .\right\Vert
_{\vartheta _{1},\vartheta _{2}}^{p,1,q,r}.$
\end{theorem}

\begin{proof}
Assume that $\left( f_{n}\right) _{n\in 
\mathbb{N}
}$ is a Cauchy sequence in $A_{\vartheta _{1},\vartheta
_{2}}^{p,1,q,r}\left( G\right) $. Thus given $\varepsilon >0,$\ there is an $%
n_{1}\in 
\mathbb{N}
$\ such that for all $n,m\geqslant n_{1}$\ implies%
\begin{eqnarray*}
\left\Vert f_{n}-f_{m}\right\Vert _{\vartheta _{1},\vartheta _{2}}^{p,1,q,r}
&=&\left\Vert f_{n}-f_{m}\right\Vert _{p1,\vartheta _{1}}+\left\Vert 
\widehat{f_{n}-f_{m}}\right\Vert _{qr,\vartheta _{2}} \\
&=&\left\Vert f_{n}-f_{m}\right\Vert _{p1,\vartheta _{1}}+\left\Vert 
\widehat{f_{n}}-\widehat{f_{m}}\right\Vert _{qr,\vartheta _{2}}<\varepsilon .
\end{eqnarray*}%
Therefore, $\left( f_{n}\right) _{n\in IN}\subset \left( L_{\vartheta
_{1}}^{p},\ell ^{1}\right) $\ and $\left( f_{n}\right) _{n\in IN}\subset
\left( L_{\vartheta _{2}}^{q},\ell ^{r}\right) $\ are Cauchy sequences with
respect to $\left\Vert .\right\Vert _{p1,\vartheta _{1}}$\ and $\left\Vert
.\right\Vert _{qr,\vartheta _{2}}$, respectively. Since the spaces $\left(
\left( L_{\vartheta _{1}}^{p},\ell ^{1}\right) ,\left\Vert .\right\Vert
_{p1,\vartheta _{1}}\right) $\ and $\left( \left( L_{\vartheta
_{2}}^{q},\ell ^{r}\right) ,\left\Vert .\right\Vert _{qr,\vartheta
_{2}}\right) $\ are two Banach spaces, there exist $f\in \left( L_{\vartheta
_{1}}^{p},\ell ^{1}\right) $ and $g\in \left( L_{\vartheta _{2}}^{q},\ell
^{r}\right) $ such that $\left\Vert f_{n}-f\right\Vert _{p1,\vartheta
_{1}}\longrightarrow 0$, $\left\Vert \widehat{f_{n}}-g\right\Vert
_{qr,\vartheta _{2}}\longrightarrow 0$. If we use the inequality $\left\Vert
.\right\Vert _{1,\vartheta _{1}}\leq \left\Vert .\right\Vert _{p1,\vartheta
_{1}},$ then we get $\left\Vert f_{n}-f\right\Vert _{1,\vartheta
_{1}}\longrightarrow 0.$ By Theorem \ref{obvious}, there is a subsequence $%
\left\{ \widehat{f_{n_{k}}}\right\} _{k\in IN}$ of $\left\{ \widehat{f_{n}}%
\right\} _{n\in IN}$ such that $\widehat{f_{n_{k}}}\longrightarrow g$ a.e.
Since $\vartheta _{1}$ is a Beurling's weight, we have%
\begin{equation*}
\left\Vert \widehat{f_{n}}-\widehat{f}\right\Vert _{\infty }\leq \left\Vert
f_{n}-f\right\Vert _{1}\leq \left\Vert f_{n}-f\right\Vert _{1,\vartheta
_{1}}.
\end{equation*}%
This follows that $\left\Vert \widehat{f_{n}}-\widehat{f}\right\Vert
_{\infty }\longrightarrow 0.$ Moreover, we get%
\begin{equation*}
\left\Vert \widehat{f_{n_{k}}}-\widehat{f}\right\Vert _{\infty
}=\sup_{k}\left\vert f_{n_{k}}\left( x\right) -f\left( x\right) \right\vert
\leq \left\Vert \widehat{f_{n}}-\widehat{f}\right\Vert _{\infty
}\longrightarrow 0.
\end{equation*}

Therefore, we have $\widehat{f}=g.$ This yields that%
\begin{equation*}
\left\Vert f_{n}-f\right\Vert _{\vartheta _{1},\vartheta
_{2}}^{p,1,q,r}=\left\Vert f_{n}-f\right\Vert _{p1,\vartheta
_{1}}+\left\Vert \widehat{f_{n}}-\widehat{f}\right\Vert _{qr,\vartheta
_{2}}\longrightarrow 0
\end{equation*}%
and $f\in A_{\vartheta _{1},\vartheta _{2}}^{p,1,q,r}\left( G\right) .$ That
is the desired result.
\end{proof}

\begin{theorem}
\label{Banach algebra}If $1<p,q,r<\infty ,$ then the space $\left(
A_{\vartheta _{1},\vartheta _{2}}^{p,1,q,r}\left( G\right) ,\left\Vert
.\right\Vert _{\vartheta _{1},\vartheta _{2}}^{p,1,q,r}\right) $ is a Banach
algebra with respect to convolution.
\end{theorem}

\begin{proof}
It is note that the space $\left( A_{\vartheta _{1},\vartheta
_{2}}^{p,1,q,r}\left( G\right) ,\left\Vert .\right\Vert _{\vartheta
_{1},\vartheta _{2}}^{p,1,q,r}\right) $ is a Banach space by the Theorem \ref%
{banach}. Now, let $f,g\in A_{\vartheta _{1},\vartheta _{2}}^{p,1,q,r}\left(
G\right) $ be given. Thus, we have $f,g\in \left( L_{\vartheta
_{1}}^{p},\ell ^{1}\right) $ and $\widehat{f},\widehat{g}\in \left(
L_{\vartheta _{2}}^{q},\ell ^{r}\right) .$ Since $\left( L_{\vartheta
_{1}}^{p},\ell ^{1}\right) $ is a Banach algebra under convolution (see \cite%
{Squ}), we have $f\ast g\in \left( L_{\vartheta _{1}}^{p},\ell ^{1}\right) $
and there exists $C\geq 1$ such that%
\begin{equation}
\left\Vert f\ast g\right\Vert _{p1,\vartheta _{1}}\leq C\left\Vert
f\right\Vert _{p1,\vartheta _{1}}\left\Vert g\right\Vert _{p1,\vartheta _{1}}%
\text{.}  \label{balgebra1}
\end{equation}

If we consider the inequality%
\begin{equation*}
\left\vert \widehat{\left( f\ast g\right) }\left( x\right) \right\vert
=\left\vert \widehat{f}\left( x\right) \right\vert \left\vert \widehat{g}%
\left( x\right) \right\vert \leq \left\Vert \widehat{f}\right\Vert _{\infty
}\left\vert \widehat{g}\left( x\right) \right\vert ,
\end{equation*}%
then we have%
\begin{equation}
\left\Vert \widehat{f\ast g}\right\Vert _{qr,\vartheta _{2}}\leq \left\Vert 
\widehat{f}\right\Vert _{\infty }\left\Vert \widehat{g}\right\Vert
_{qr,\vartheta _{2}}\leq \left\Vert f\right\Vert _{p1,\vartheta
_{1}}\left\Vert \widehat{g}\right\Vert _{qr,\vartheta _{2}}
\label{balgebra3}
\end{equation}%
and $\widehat{f\ast g}\in \left( L_{\vartheta _{2}}^{q},\ell ^{r}\right) .$
Therefore, we have $f\ast g\in A_{\vartheta _{1},\vartheta
_{2}}^{p,1,q,r}\left( G\right) .$ This follows by (\ref{balgebra1}) and (\ref%
{balgebra3}) that%
\begin{eqnarray*}
\left\Vert f\ast g\right\Vert _{\vartheta _{1},\vartheta _{2}}^{p,1,q,r}
&=&\left\Vert f\ast g\right\Vert _{p1,\vartheta _{1}}+\left\Vert \widehat{%
f\ast g}\right\Vert _{qr,\vartheta _{2}} \\
&\leq &C\left\Vert f\right\Vert _{p1,\vartheta _{1}}\left( \left\Vert
g\right\Vert _{p1,\vartheta _{1}}+\left\Vert \widehat{g}\right\Vert
_{qr,\vartheta _{2}}\right) \\
&\leq &C\left\Vert f\right\Vert _{\vartheta _{1},\vartheta
_{2}}^{p,1,q,r}\left\Vert g\right\Vert _{\vartheta _{1},\vartheta
_{2}}^{p,1,q,r}.
\end{eqnarray*}
\end{proof}

\begin{theorem}
The space $\left( A_{\vartheta _{1},\vartheta _{2}}^{p,1,q,r}\left( G\right)
,\left\Vert .\right\Vert _{\vartheta _{1},\vartheta _{2}}^{p,1,q,r}\right) $
is a BF-space.
\end{theorem}

\begin{proof}
Assume that $K\subset G$ is a compact subset and let $f\in A_{\vartheta
_{1},\vartheta _{2}}^{p,1,q,r}\left( G\right) $. Since $\vartheta _{1}$ is a
Beurling's weight function and the space $L_{\vartheta _{1}}^{p}\left(
G\right) $ is continuously embedded in $L_{\vartheta _{1}}^{1}\left(
G\right) $, we have%
\begin{eqnarray*}
\int\limits_{K}\left\vert f(x)\right\vert dx &\leq &\left\Vert f\right\Vert
_{1}\leq \left\Vert f\right\Vert _{1,\vartheta _{1}} \\
&\leq &\left\Vert f\right\Vert _{p1,\vartheta _{1}}\leq \left\Vert
f\right\Vert _{\vartheta _{1},\vartheta _{2}}^{p,1,q,r}\text{.}
\end{eqnarray*}%
This completes the proof.
\end{proof}

\begin{theorem}
\label{statement}The following statements hold.

\begin{enumerate}
\item[\textit{(i)}] The space $A_{\vartheta _{1},\vartheta
_{2}}^{p,1,q,r}\left( G\right) $ is translation invariant and for every $%
f\in A_{\vartheta _{1},\vartheta _{2}}^{p,1,q,r}\left( G\right) $ and $y\in
G $ the inequality $C_{1}\left( f\right) \vartheta _{1}\left( y\right) \leq
\left\Vert T_{y}f\right\Vert _{\vartheta _{1},\vartheta _{2}}^{p,1,q,r}\leq
C_{2}\left( f\right) \vartheta _{1}\left( y\right) $ holds where $%
C_{1}\left( f\right) >0$ and $C_{2}\left( f\right) =\left\Vert f\right\Vert
_{\vartheta _{1},\vartheta _{2}}^{p,1,q,r}$.

\item[\textit{(ii)}] The map $y\longrightarrow T_{y}f$ is continuous from $G$
into $A_{\vartheta _{1},\vartheta _{2}}^{p,1,q,r}\left( G\right) $ for every 
$f\in A_{\vartheta _{1},\vartheta _{2}}^{p,1,q,r}\left( G\right) .$
\end{enumerate}
\end{theorem}

\begin{proof}
Let $f\in A_{\vartheta _{1},\vartheta _{2}}^{p,1,q,r}\left( G\right) $.
Since $\vartheta _{1}$ is a Beurling's weight function, it is easy to see
that $T_{y}f\in \left( L_{\vartheta _{1}}^{p},\ell ^{1}\right) $ and $%
\left\Vert T_{y}f\right\Vert _{p1,\vartheta _{1}}\leq \vartheta _{1}\left(
y\right) \left\Vert f\right\Vert _{p1,\vartheta _{1}}$ for all $y\in G$.

Moreover, for $p>1,$ it is clear that $\left( L_{\vartheta _{1}}^{p},\ell
^{1}\right) $ is continuously embedded in $L_{\vartheta _{1}}^{p}\left(
G\right) ,$ see \cite{Squ}. Thus, if we consider the Lemma 2.2 in \cite%
{Feich3}, then there is a constant $C>0$ depending on $f$ such that%
\begin{equation*}
C\vartheta _{1}\left( y\right) \leq \left\Vert T_{y}f\right\Vert
_{p,\vartheta _{1}}\leq C^{\ast }\left\Vert T_{y}f\right\Vert _{p1,\vartheta
_{1}}\leq C^{\ast }\left\Vert T_{y}f\right\Vert _{\vartheta _{1},\vartheta
_{2}}^{p,1,q,r}
\end{equation*}%
for all $y\in G.$ This follows that%
\begin{equation}
C_{1}\vartheta _{1}\left( y\right) \leq \left\Vert T_{y}f\right\Vert
_{\vartheta _{1},\vartheta _{2}}^{p,1,q,r}  \label{sta1}
\end{equation}%
where $C_{1}=\frac{C}{C^{\ast }}$ depends on $f.$ It is clear that $\widehat{%
L_{y}f}=M_{-y}\widehat{f}.$ Moreover, if we consider that the weighted
amalgam space $\left( L_{\vartheta _{2}}^{q},\ell ^{r}\right) $ is strongly
character invariant and the function $t\longrightarrow M_{t}f$ is continuous
from $\widehat{G}$ into $\left( L_{\vartheta _{2}}^{q},\ell ^{r}\right) $
(see \cite{A1}, \cite{Pan}, \cite{Sag}), then we have%
\begin{equation*}
\left\Vert \widehat{T_{y}f}\right\Vert _{qr,\vartheta _{2}}=\left\Vert M_{-y}%
\widehat{f}\right\Vert _{qr,\vartheta _{2}}=\left\Vert \widehat{f}%
\right\Vert _{qr,\vartheta _{2}}<\infty .
\end{equation*}

This follows that $T_{y}f\in A_{\vartheta _{1},\vartheta
_{2}}^{p,1,q,r}\left( G\right) .$ Since $\vartheta _{1}\geq 1,$ we get%
\begin{eqnarray}
\left\Vert T_{y}f\right\Vert _{\vartheta _{1},\vartheta _{2}}^{p,1,q,r}
&\leq &\vartheta _{1}\left( y\right) \left\Vert f\right\Vert _{p1,\vartheta
_{1}}+\left\Vert \widehat{f}\right\Vert _{qr,\vartheta _{2}}  \notag \\
&\leq &\vartheta _{1}\left( y\right) \left\Vert f\right\Vert _{\vartheta
_{1},\vartheta _{2}}^{p,1,q,r}  \label{sta2}
\end{eqnarray}%
for all $y\in G.$ This completes \textit{(i)} by (\ref{sta1}) and (\ref{sta2}%
)\textit{. }It is obvious that $T_{y}$ is linear. For any $\varepsilon >0,$
there is a neighbourhood $V_{1}$ of the unit element of $G$ such that%
\begin{equation}
\left\Vert T_{y}f-f\right\Vert _{p1,\vartheta _{1}}<\frac{\varepsilon }{2}
\label{trans1}
\end{equation}%
for all $y\in V_{1}.$ Also, there is a neighbourhood $V_{2}$ of the unit
element of $G$ such that%
\begin{equation}
\left\Vert \widehat{T_{y}f}-\widehat{f}\right\Vert _{qr,\vartheta
_{2}}=\left\Vert M_{-y}\widehat{f}-\widehat{f}\right\Vert _{qr,\vartheta
_{2}}<\frac{\varepsilon }{2}  \label{trans2}
\end{equation}%
for all $y\in V_{2}.$ Now, let us denote $U=V_{1}\cap V_{2}.$ By (\ref%
{trans1}) and (\ref{trans2}), we have%
\begin{equation*}
\left\Vert T_{y}f-f\right\Vert _{\vartheta _{1},\vartheta
_{2}}^{p,1,q,r}=\left\Vert T_{y}f-f\right\Vert _{p1,\vartheta
_{1}}+\left\Vert \widehat{T_{y}f-f}\right\Vert _{qr,\vartheta
_{2}}<\varepsilon
\end{equation*}%
for all $y\in U.$ That is the desired result.
\end{proof}

\begin{theorem}
\label{segal}Assume that $\vartheta _{1}$ satisfies (BD) condition. Then the
space $A_{\vartheta _{1},\vartheta _{2}}^{1,1,q,r}\left( G\right) $ is a $%
S_{\vartheta _{1}}$ algebra.
\end{theorem}

\begin{proof}
We have already proved the several conditions for $S_{\vartheta _{1}}$
algebra\textbf{\ }in Theorem \ref{Banach algebra} and Theorem \ref{statement}%
. Now, let us denote%
\begin{equation*}
F_{0,\vartheta _{1}}=\left\{ f\in L_{\vartheta _{1}}^{1}\left( G\right) :%
\widehat{f}\text{ has a compact support}\right\} .
\end{equation*}%
Since $\vartheta _{1}$ satisfies (BD) condition, the set $F_{0,\vartheta
_{1}}$ is dense in $L_{\vartheta _{1}}^{1}\left( G\right) .$ It is clear
that $C_{c}\left( G\right) \subset \left( L_{\vartheta _{2}}^{q},\ell
^{r}\right) .$ Because of the fact that the inclusions $F_{0,\vartheta
_{1}}\subset A_{\vartheta _{1},\vartheta _{2}}^{1,1,q,r}\left( G\right)
\subset L_{\vartheta _{1}}^{1}\left( G\right) $ hold and $F_{0,\vartheta
_{1}}$ is dense in $L_{\vartheta _{1}}^{1}\left( G\right) $, then $%
A_{\vartheta _{1},\vartheta _{2}}^{1,1,q,r}\left( G\right) $ is dense in $%
L_{\vartheta _{1}}^{1}\left( G\right) $. That is the desired result.
\end{proof}

\begin{theorem}
\label{module}Let $1<p,q,r<\infty .$ If $\vartheta _{1}\prec \vartheta _{0}$%
, then $A_{\vartheta _{1},\vartheta _{2}}^{p,1,q,r}\left( G\right) $ is a
Banach $L_{\vartheta _{0}}^{1}\left( G\right) $-module with respect to $%
\left\Vert .\right\Vert _{\vartheta _{1},\vartheta _{2}}^{p,1,q,r}$.
\end{theorem}

\begin{proof}
Since $\vartheta _{1}\prec \vartheta _{0}$, we have $L_{\vartheta
_{0}}^{1}\left( G\right) \hookrightarrow L_{\vartheta _{1}}^{1}\left(
G\right) .$ Moreover, if we consider the fact that $\left( L_{\vartheta
_{1}}^{p},\ell ^{1}\right) $ is a Banach $L_{\vartheta _{1}}^{1}\left(
G\right) $-module for $1<p<\infty $, then there exists $C>0$ such that%
\begin{eqnarray*}
\left\Vert f\ast g\right\Vert _{\vartheta _{1},\vartheta _{2}}^{p,1,q,r}
&\leq &C\left\Vert f\right\Vert _{p1,\vartheta _{1}}\left\Vert g\right\Vert
_{1,\vartheta _{1}}+\left\Vert \widehat{f}\right\Vert _{qr,\vartheta
_{2}}\left\Vert \widehat{g}\right\Vert _{\infty } \\
&\leq &C\left\Vert f\right\Vert _{p1,\vartheta _{1}}\left\Vert g\right\Vert
_{1,\vartheta _{0}}+\left\Vert \widehat{f}\right\Vert _{qr,\vartheta
_{2}}\left\Vert \widehat{g}\right\Vert _{\infty } \\
&\leq &C\left\Vert f\right\Vert _{p1,\vartheta _{1}}\left\Vert g\right\Vert
_{1,\vartheta _{0}}+\left\Vert \widehat{f}\right\Vert _{qr,\vartheta
_{2}}\left\Vert g\right\Vert _{1,\vartheta _{0}} \\
&\leq &\max \left\{ 1,C\right\} \left\Vert f\right\Vert _{\vartheta
_{1},\vartheta _{2}}^{p,1,q,r}\left\Vert g\right\Vert _{1,\vartheta _{0}}
\end{eqnarray*}%
for any $f\in A_{\vartheta _{1},\vartheta _{2}}^{p,1,q,r}\left( G\right) $
and $g\in L_{\vartheta _{0}}^{1}\left( G\right) .$ If we define a new norm $%
\left\Vert \left\vert .\right\vert \right\Vert $ on $L_{\vartheta
_{0}}^{1}\left( G\right) $ such that $\left\Vert \left\vert .\right\vert
\right\Vert =\max \left\{ c_{1},c_{2}\right\} \left\Vert .\right\Vert
_{1,\vartheta _{0}},$ then this norm is equivalent to the norm $\left\Vert
.\right\Vert _{1,\vartheta _{0}}$ on $L_{\vartheta _{0}}^{1}\left( G\right)
. $ This completes the proof.
\end{proof}

\begin{theorem}
Suppose that $\vartheta _{1}$ satisfies (BD) condition. Then $A_{\vartheta
_{1},\vartheta _{2}}^{p,1,q,r}\left( G\right) $ has an approximate unit with
compactly supported Fourier transforms.
\end{theorem}

\begin{proof}
Let $f\in A_{\vartheta _{1},\vartheta _{2}}^{p,1,q,r}\left( G\right) $\ and $%
\varepsilon >0$\ be given. Then by Theorem \ref{statement}, there exists a
neighbourhood $U$\ of the unit element of $G$\ such that 
\begin{equation}
\left\Vert T_{y}f-f\right\Vert _{\vartheta _{1},\vartheta _{2}}^{p,1,q,r}<%
\frac{\varepsilon }{2}  \label{5}
\end{equation}%
for all $y\in U$. Let be taken a\ non-negative function $g\in C_{c}(G)$\
such that supp$g\subset U$\ and $\tint\limits_{G}g(y)dy=1$. Since%
\begin{equation*}
g\ast f-f=\tint\limits_{G}g(y)\left( T_{y}f-f\right) dy
\end{equation*}%
by using the inequality (\ref{5}), we get%
\begin{eqnarray}
\left\Vert g\ast f-f\right\Vert _{\vartheta _{1},\vartheta _{2}}^{p,1,q,r}
&=&\left\Vert \tint\limits_{G}g(y)\left( T_{y}f-f\right) dy\right\Vert
_{\vartheta _{1},\vartheta _{2}}^{p,1,q,r}  \notag \\
&\leq &\tint\limits_{G}g(y)\left\Vert T_{y}f-f\right\Vert _{\vartheta
_{1},\vartheta _{2}}^{p,1,q,r}dy  \notag \\
&=&\tint\limits_{U}g(y)\left\Vert T_{y}f-f\right\Vert _{\vartheta
_{1},\vartheta _{2}}^{p,1,q,r}dy<\frac{\varepsilon }{2}\tint%
\limits_{G}g(y)dy=\frac{\varepsilon }{2}.  \label{6}
\end{eqnarray}%
Moreover, since the set $F_{0,\vartheta _{1}}$\ is dense in $L_{\vartheta
_{1}}^{1}\left( G\right) $\ by Theorem \ref{segal}, there exists $h\in
F_{0,\vartheta _{1}}$\ such that 
\begin{equation}
\left\Vert g-h\right\Vert _{1,\vartheta _{1}}<\frac{\varepsilon }{%
2\left\Vert f\right\Vert _{\vartheta _{1},\vartheta _{2}}^{p,1,q,r}}.
\label{9}
\end{equation}%
If we consider that the space $C_{c}(G)$ is included in all amalgam spaces,
then it we have $h\in A_{\vartheta _{1},\vartheta _{2}}^{p,1,q,r}(G)$. Also,
it is clear that%
\begin{equation}
\left\Vert \left( h-g\right) \ast f\right\Vert _{\vartheta _{1},\vartheta
_{2}}^{p,1,q,r}\leq \left\Vert \left( h-g\right) \right\Vert _{1,\vartheta
_{1}}\left\Vert f\right\Vert _{\vartheta _{1},\vartheta _{2}}^{p,1,q,r}.
\label{8}
\end{equation}%
By (\ref{6}), (\ref{9}) and (\ref{8}), we obtain%
\begin{eqnarray*}
\left\Vert h\ast f-f\right\Vert _{\vartheta _{1},\vartheta _{2}}^{p,1,q,r}
&\leq &\left\Vert h\ast f-g\ast f\right\Vert _{\vartheta _{1},\vartheta
_{2}}^{p,1,q,r}+\left\Vert g\ast f-f\right\Vert _{\vartheta _{1},\vartheta
_{2}}^{p,1,q,r} \\
&<&\frac{\varepsilon }{2\left\Vert f\right\Vert _{\vartheta _{1},\vartheta
_{2}}^{p,1,q,r}}\left\Vert f\right\Vert _{\vartheta _{1},\vartheta
_{2}}^{p,1,q,r}+\frac{\varepsilon }{2}=\varepsilon \text{.}
\end{eqnarray*}
\end{proof}

Consider the mapping $\Phi $ from $A_{\vartheta _{1},\vartheta
_{2}}^{p,1,q,r}\left( G\right) $ into $\left( L_{\vartheta _{1}}^{p},\ell
^{1}\right) \times \left( L_{\vartheta _{2}}^{q},\ell ^{r}\right) $ defined
by $\Phi (f)=\left( f,\widehat{f}\right) $. This is a linear isometry from $%
A_{\vartheta _{1},\vartheta _{2}}^{p,1,q,r}\left( G\right) $ into $\left(
L_{\vartheta _{1}}^{p},\ell ^{1}\right) \times \left( L_{\vartheta
_{2}}^{q},\ell ^{r}\right) $ in sense to the norm%
\begin{equation*}
\left\Vert \left\vert \left( f,\widehat{f}\right) \right\vert \right\Vert
=\left\Vert f\right\Vert _{p1,\vartheta _{1}}+\left\Vert \widehat{f}%
\right\Vert _{qr,\vartheta _{2}}
\end{equation*}%
for $f\in A_{\vartheta _{1},\vartheta _{2}}^{p,1,q,r}\left( G\right) .$
Hence it is easy to see that $A_{\vartheta _{1},\vartheta
_{2}}^{p,1,q,r}\left( G\right) $ is a closed subspace of the Banach space $%
\left( L_{\vartheta _{1}}^{p},\ell ^{1}\right) \times \left( L_{\vartheta
_{2}}^{q},\ell ^{r}\right) $. Let%
\begin{equation*}
H=\left\{ \left( f,\widehat{f}\right) :f\in A_{\vartheta _{1},\vartheta
_{2}}^{p,1,q,r}\left( G\right) \right\}
\end{equation*}%
and%
\begin{equation*}
K=\left\{ 
\begin{array}{c}
\left( \varphi ,\psi \right) :\left( \varphi ,\psi \right) \in \left(
L_{\vartheta _{1}^{-1}}^{p^{\prime }},\ell ^{\infty }\right) \times \left(
L_{\vartheta _{2}^{-1}}^{q^{\prime }},\ell ^{r^{\prime }}\right) , \\ 
\int\limits_{G}f(x)\varphi (x)dx+\int\limits_{G}\widehat{f}(y)\psi (y)dy=0%
\text{, for all }\left( f,\widehat{f}\right) \in H%
\end{array}%
\right\} \text{,}
\end{equation*}%
where $\frac{1}{r}+\frac{1}{r^{\prime }}=1,\frac{1}{p}+\frac{1}{p^{\prime }}%
=1$ and $\frac{1}{q}+\frac{1}{q^{\prime }}=1$.

The following theorem is easily proved by Duality Theorem 1.7 in \cite{Liu1}.

\begin{theorem}
\label{isomorphic}The dual space $\left( A_{\vartheta _{1},\vartheta
_{2}}^{p,1,q,r}\left( G\right) \right) ^{\ast }$ of $A_{\vartheta
_{1},\vartheta _{2}}^{p,1,q,r}\left( G\right) $ is isomorphic to $\left(
L_{\vartheta _{1}^{-1}}^{p^{\prime }},\ell ^{\infty }\right) \times \left(
L_{\vartheta _{2}^{-1}}^{q^{\prime }},\ell ^{r^{\prime }}\right) /K$.
\end{theorem}

\section{\textbf{Inclusions of the spaces }$A_{\protect\vartheta _{1},%
\protect\vartheta _{2}}^{p,1,q,r}\left( G\right) $}

\begin{theorem}
\label{closedgraph}The inclusion $A_{\vartheta _{1},\vartheta
_{2}}^{p,1,q,r}\left( G\right) \subset A_{\vartheta _{3},\vartheta
_{4}}^{p,1,q,r}\left( G\right) $ holds if and only if the space $%
A_{\vartheta _{1},\vartheta _{2}}^{p,1,q,r}\left( G\right) $ is continuously
embedded in $A_{\vartheta _{3},\vartheta _{4}}^{p,1,q,r}\left( G\right) $.
\end{theorem}

\begin{proof}
The sufficient condition of the theorem is clear by definition of embedding.
Now, assume that $A_{\vartheta _{1},\vartheta _{2}}^{p,1,q,r}\left( G\right)
\subset A_{\vartheta _{3},\vartheta _{4}}^{p,1,q,r}\left( G\right) $ holds.
Moreover, we define the sum norm $\left\Vert \left\vert .\right\vert
\right\Vert =\left\Vert .\right\Vert _{\vartheta _{1},\vartheta
_{2}}^{p,1,q,r}+\left\Vert .\right\Vert _{\vartheta _{3},\vartheta
_{4}}^{p,1,q,r}.$ It is easy to see that $\left( A_{\vartheta _{1},\vartheta
_{2}}^{p,1,q,r}\left( G\right) ,\left\Vert \left\vert .\right\vert
\right\Vert \right) $ is a Banach space. Now, let us define the unit
function $I$\ from $\left( A_{\vartheta _{1},\vartheta _{2}}^{p,1,q,r}\left(
G\right) ,\left\Vert \left\vert .\right\vert \right\Vert \right) $\ into $%
\left( A_{\vartheta _{1},\vartheta _{2}}^{p,1,q,r}\left( G\right)
,\left\Vert .\right\Vert _{\vartheta _{1},\vartheta _{2}}^{p,1,q,r}\right) $%
. Since the inequality%
\begin{equation*}
\left\Vert I\left( f\right) \right\Vert _{\vartheta _{1},\vartheta
_{2}}^{p,1,q,r}=\left\Vert f\right\Vert _{\vartheta _{1},\vartheta
_{2}}^{p,1,q,r}\leq \left\Vert \left\vert f\right\vert \right\Vert
\end{equation*}%
is satisfied, $I$\ is continuous. If we consider the Banach's theorem, then $%
I$\ is a homeomorphism, see \cite{Cartan}. That means the norms $\left\Vert
\left\vert .\right\vert \right\Vert $\ and $\left\Vert .\right\Vert
_{\vartheta _{1},\vartheta _{2}}^{p,1,q,r}$\ are equivalent. Thus, for every 
$f\in A_{\vartheta _{1},\vartheta _{2}}^{p,1,q,r}\left( G\right) $\ there
exists $c>0$\ such that%
\begin{equation}
\left\Vert \left\vert f\right\vert \right\Vert \leq c\left\Vert f\right\Vert
_{\vartheta _{1},\vartheta _{2}}^{p,1,q,r}.  \label{embedding4}
\end{equation}%
Therefore, by using (\ref{embedding4}) and the definition of norm $%
\left\Vert \left\vert .\right\vert \right\Vert ,$\ we obtain 
\begin{equation*}
\left\Vert f\right\Vert _{\vartheta _{3},\vartheta _{4}}^{p,1,q,r}\leq
\left\Vert \left\vert f\right\vert \right\Vert \leq c\left\Vert f\right\Vert
_{\vartheta _{1},\vartheta _{2}}^{p,1,q,r}\text{.}
\end{equation*}%
That is the desired result.
\end{proof}

Now, we give some continuous embeddings of $A_{\vartheta _{1},\vartheta
_{2}}^{p,1,q,r}\left( G\right) $ under some conditions.

\begin{theorem}
\label{true}The following statements are true.

\begin{enumerate}
\item[\textit{(i)}] If $p_{2}\leq p_{1}$ and $r_{1}\leq r_{2}$, then we have 
$A_{\vartheta _{1},\vartheta _{2}}^{p_{1},1,q,r_{1}}\left( G\right)
\hookrightarrow A_{\vartheta _{1},\vartheta _{2}}^{p_{2},1,q,r_{2}}\left(
G\right) .$

\item[\textit{(ii)}] If $\vartheta _{3}\prec \vartheta _{1}$ and $\vartheta
_{4}\prec \vartheta _{2},$ then $A_{\vartheta _{1},\vartheta
_{2}}^{p,1,q,r}\left( G\right) $ is continuously embedded in $A_{\vartheta
_{3},\vartheta _{4}}^{p,1,q,r}\left( G\right) $.

\item[\textit{(iii)}] If $p_{2}\leq p_{1}$ and $q_{2}\leq q_{1}$, then we
get $A_{\vartheta _{1},\vartheta _{2}}^{p_{1},1,q_{1},r}\left( G\right)
\hookrightarrow A_{\vartheta _{1},\vartheta _{2}}^{p_{2},1,q_{2},r}\left(
G\right) .$

\item[\textit{(iv)}] If $\vartheta _{3}\prec \vartheta _{1}$, $q_{2}\leq
q_{1}$ and $r_{1}\leq r_{2},$ then the space $A_{\vartheta _{1},\vartheta
_{2}}^{p,1,q_{1},r_{1}}\left( G\right) $ is continuously embedded in $%
A_{\vartheta _{3},\vartheta _{2}}^{p,1,q_{2},r_{2}}\left( G\right) .$
\end{enumerate}
\end{theorem}

\begin{proof}
Let $p_{2}\leq p_{1}$ and $r_{1}\leq r_{2}$. Moreover, since the embeddings $%
\left( L_{\vartheta _{1}}^{p_{1}},\ell ^{1}\right) \hookrightarrow \left(
L_{\vartheta _{1}}^{p_{2}},\ell ^{1}\right) $ and $\left( L_{\vartheta
_{2}}^{q},\ell ^{r_{1}}\right) \hookrightarrow \left( L_{\vartheta
_{2}}^{q},\ell ^{r_{2}}\right) $ hold under these hypotheses (see \cite%
{Heil1}, \cite{Squ}), there exist $c_{1},c_{2}>0$ such that%
\begin{eqnarray*}
\left\Vert f\right\Vert _{\vartheta _{1},\vartheta _{2}}^{p_{2},1,q,r_{2}}
&\leq &c_{1}\left\Vert f\right\Vert _{p_{1}1,\vartheta _{1}}+c_{2}\left\Vert 
\widehat{f}\right\Vert _{qr_{1},\vartheta _{2}} \\
&\leq &\max \left\{ c_{1},c_{2}\right\} \left\Vert f\right\Vert _{\vartheta
_{1},\vartheta _{2}}^{p_{1},1,q,r_{1}}
\end{eqnarray*}%
for all $f\in A_{\vartheta _{1},\vartheta _{2}}^{p_{1},1,q,r_{1}}\left(
G\right) $. This completes $\left( i\right) .$ Now, let $f\in A_{\vartheta
_{1},\vartheta _{2}}^{p,1,q,r}\left( G\right) $ be given. Hence, we get $%
f\in $ $\left( L_{\vartheta _{1}}^{p},\ell ^{1}\right) $ and $\widehat{f}\in
\left( L_{\vartheta _{2}}^{q},\ell ^{r}\right) $. Moreover, assume that $%
\vartheta _{3}\prec \vartheta _{1}$ and $\vartheta _{4}\prec \vartheta _{2}$
hold. This follows that there exist $c_{3},c_{4}>0$ such that $\vartheta
_{3}\left( x\right) \leq c_{3}\vartheta _{1}\left( x\right) $ and $\vartheta
_{4}\left( x\right) \leq c_{4}\vartheta _{2}\left( x\right) $ for all $x\in
G.$ Therefore, we have%
\begin{eqnarray*}
\left\Vert f\right\Vert _{\vartheta _{3},\vartheta _{4}}^{p,1,q,r}
&=&\dsum\limits_{n\in J}\left\Vert f\right\Vert _{L_{\vartheta
_{3}}^{p}\left( I_{n}\right) }+\left( \dsum\limits_{m\in J}\left\Vert 
\widehat{f}\right\Vert _{L_{\vartheta _{4}}^{q}\left( I_{m}\right)
}^{r}\right) ^{\frac{1}{r}} \\
&\leq &c_{3}\dsum\limits_{n\in J}\left\Vert f\right\Vert _{L_{\vartheta
_{1}}^{p}\left( I_{n}\right) }+c_{4}\left( \dsum\limits_{m\in J}\left\Vert 
\widehat{f}\right\Vert _{L_{\vartheta _{2}}^{q}\left( I_{m}\right)
}^{r}\right) ^{\frac{1}{r}} \\
&\leq &\max \left\{ c_{3},c_{4}\right\} \left\Vert f\right\Vert _{\vartheta
_{1},\vartheta _{2}}^{p,1,q,r}.
\end{eqnarray*}

This proves $(ii)$. If we consider $(i)$ and $(ii),$ then we obtain $(iii)$
and $(iv).$
\end{proof}

\begin{theorem}
\label{ifandonlyif}Let $\vartheta ,\vartheta _{1}$ and $\vartheta _{2}$ be
Beurling's weights. Then the space $A_{\vartheta _{1},\vartheta
}^{p,1,q,r}\left( G\right) $ is continuously embedded in $A_{\vartheta
_{2},\vartheta }^{p,1,q,r}\left( G\right) $ if and only if $\vartheta
_{2}\prec \vartheta _{1}.$
\end{theorem}

\begin{proof}
Assume that $A_{\vartheta _{1},\vartheta }^{p,1,q,r}\left( G\right)
\hookrightarrow A_{\vartheta _{2},\vartheta }^{p,1,q,r}\left( G\right) .$ By
the Theorem \ref{statement}, there are $C_{1},C_{2},C_{3},C_{4}>0$ such that%
\begin{equation}
C_{1}\vartheta _{1}\left( y\right) \leq \left\Vert T_{y}f\right\Vert
_{\vartheta _{1},\vartheta }^{p,1,q,r}\leq C_{2}\vartheta _{1}\left( y\right)
\label{embed1}
\end{equation}%
and%
\begin{equation}
C_{3}\vartheta _{2}\left( y\right) \leq \left\Vert T_{y}f\right\Vert
_{\vartheta _{2},\vartheta }^{p,1,q,r}\leq C_{4}\vartheta _{2}\left( y\right)
\label{embed2}
\end{equation}%
for $y\in G.$ Since $T_{y}f\in A_{\vartheta _{1},\vartheta }^{p,1,q,r}\left(
G\right) $ for every $f\in A_{\vartheta _{1},\vartheta }^{p,1,q,r}\left(
G\right) ,$ there is a $C>0$ such that%
\begin{equation}
\left\Vert T_{y}f\right\Vert _{\vartheta _{2},\vartheta }^{p,1,q,r}\leq
C\left\Vert T_{y}f\right\Vert _{\vartheta _{1},\vartheta }^{p,1,q,r}.
\label{embed3}
\end{equation}%
If we consider the (\ref{embed1}), (\ref{embed2}) and (\ref{embed3}), then
we have%
\begin{equation*}
C_{3}\vartheta _{2}\left( y\right) \leq \left\Vert T_{y}f\right\Vert
_{\vartheta _{2},\vartheta }^{p,1,q,r}\leq C\left\Vert T_{y}f\right\Vert
_{\vartheta _{1},\vartheta }^{p,1,q,r}\leq CC_{2}\vartheta _{1}\left(
y\right) .
\end{equation*}%
This follows that $\vartheta _{2}\prec \vartheta _{1}.$ The rest of the
proof can be proven by the similar in Theorem \ref{true}.
\end{proof}

The following corollary can be easily proven by using the Theorem \ref{true}
and Theorem \ref{ifandonlyif}.

\begin{corollary}
The following expressions are true.
\end{corollary}

\begin{enumerate}
\item[\textit{(i)}] The equality $A_{\vartheta _{1},\vartheta
_{2}}^{p,1,q,r}\left( G\right) =A_{\vartheta _{3},\vartheta
_{4}}^{p,1,q,r}\left( G\right) $ is satisfied if $\vartheta _{1}\thickapprox
\vartheta _{3}$ and $\vartheta _{2}\thickapprox \vartheta _{4}.$

\item[\textit{(ii)}] If $p_{2}\leq p_{1}$, $q_{2}\leq q_{1}$ and $r_{1}\leq
r_{2},$ then we get $A_{\vartheta _{1},\vartheta
_{2}}^{p_{1},1,q_{1},r_{1}}\left( G\right) \hookrightarrow A_{\vartheta
_{1},\vartheta _{2}}^{p_{1},1,q_{2},r_{2}}\left( G\right) .$

\item[\textit{(iii)}] Let $p_{2}\leq p_{1}$, $q_{2}\leq q_{1}$ and $%
r_{1}\leq r_{2}.$ If $\vartheta _{3}\prec \vartheta _{1}$ and $\vartheta
_{4}\prec \vartheta _{2},$ then we have $A_{\vartheta _{1},\vartheta
_{2}}^{p_{1},1,q_{1},r_{1}}\left( G\right) \hookrightarrow A_{\vartheta
_{3},\vartheta _{4}}^{p_{1},1,q_{2},r_{2}}\left( G\right) $.

\item[\textit{(iv)}] The expression $A_{\vartheta _{1},\vartheta
}^{p,1,q,r}\left( G\right) =A_{\vartheta _{2},\vartheta }^{p,1,q,r}\left(
G\right) $ holds if and only if $\vartheta _{2}\thickapprox \vartheta _{1}.$
\end{enumerate}

\section{Compact Embeddings of the Space $A_{\protect\vartheta _{1},\protect%
\vartheta _{2}}^{p,1,q,r}\left( 
\mathbb{R}
^{d}\right) $}

Now, we investigate compact embeddings of the spaces $A_{\vartheta
_{1},\vartheta _{2}}^{p,1,q,r}\left( G\right) $ with the similar methods in 
\cite{Gur2}. Also, we will take $G=%
\mathbb{R}
^{d}$ with Lebesgue measure $dx$ for compact embeddings.

\begin{theorem}
\label{vague}Let $1<p,q,r<\infty .$ Assume that $\left( f_{n}\right) _{n\in 
\mathbb{N}
}$ is a sequence in $A_{\vartheta _{1},\vartheta _{2}}^{p,1,q,r}\left( 
\mathbb{R}
^{d}\right) $. If $\left( f_{n}\right) _{n\in 
\mathbb{N}
}$ converges to zero in $A_{\vartheta _{1},\vartheta _{2}}^{p,1,q,r}\left( 
\mathbb{R}
^{d}\right) $, then $\left( f_{n}\right) _{n\in 
\mathbb{N}
}$ converges to zero in the vague topology (which means that%
\begin{equation*}
\int\limits_{%
\mathbb{R}
^{d}}f_{n}(x)k(x)dx\longrightarrow 0
\end{equation*}%
for $n\longrightarrow \infty $ for all $k\in C_{c}(%
\mathbb{R}
^{d})$,see \cite{Dieu}).
\end{theorem}

\begin{proof}
Let $k\in C_{c}(%
\mathbb{R}
^{d})$. For $p>1,$ since the space $\left( L_{\vartheta _{1}}^{p},\ell
^{1}\right) $ is continuously embedded in $L_{\vartheta _{1}}^{p}\left( 
\mathbb{R}
^{d}\right) ,$ we have%
\begin{eqnarray*}
\left\vert \dint\limits_{%
\mathbb{R}
^{d}}f_{n}(x)k(x)dx\right\vert &\leq &\left\Vert k\right\Vert _{p^{\prime
},\vartheta _{1}}\left\Vert f_{n}\right\Vert _{p,\vartheta _{1}}\leq
C\left\Vert k\right\Vert _{p^{\prime },\vartheta _{1}}\left\Vert
f_{n}\right\Vert _{p1,\vartheta _{1}} \\
&\leq &\left\Vert k\right\Vert _{p^{\prime },\vartheta _{1}}\left\Vert
f_{n}\right\Vert _{\vartheta _{1},\vartheta _{2}}^{p,1,q,r}
\end{eqnarray*}%
where $\frac{1}{p^{\prime }}+\frac{1}{p}=1$ by the H\"{o}lder inequality.
Therefore, the sequence $\left( f_{n}\right) _{n\in 
\mathbb{N}
}$ converges to zero in vague topology.
\end{proof}

\begin{theorem}
\label{tendtozero}If $\vartheta \prec \vartheta _{1}$ and $\frac{\vartheta
\left( x\right) }{\vartheta _{1}\left( x\right) }$ doesn't tend to zero in $%
\mathbb{R}
^{d}$ as $x\longrightarrow \infty ,$ then the embedding of the space $%
A_{\vartheta _{1},\vartheta _{2}}^{p,1,q,r}\left( 
\mathbb{R}
^{d}\right) $ into $\left( L_{\vartheta }^{p},\ell ^{1}\right) $ is never
compact.
\end{theorem}

\begin{proof}
Since $\vartheta \prec \vartheta _{1}$, there exists $C_{1}>0$ such that $%
\vartheta (x)\leq C_{1}\vartheta _{1}(x)$ for all $x\in 
\mathbb{R}
^{d}$. This follows that $A_{\vartheta _{1},\vartheta _{2}}^{p,1,q,r}\left( 
\mathbb{R}
^{d}\right) \subset \left( L_{\vartheta }^{p},\ell ^{1}\right) $. Assume
that $\left( t_{n}\right) _{n\in 
\mathbb{N}
}$ is a sequence with $t_{n}\longrightarrow \infty $ in $%
\mathbb{R}
^{d}$. Since $\frac{\vartheta (x)}{\vartheta _{1}\left( x\right) }$ does not
tend to zero as $x\longrightarrow \infty $, there exists $\delta >0$ such
that $\frac{\vartheta (x)}{\vartheta _{1}\left( x\right) }\geq \delta >0$
for $x\longrightarrow \infty .$ To end the proof, we take any fixed $f\in
A_{\vartheta _{1},\vartheta _{2}}^{p,1,q,r}\left( 
\mathbb{R}
^{d}\right) $ and define a sequence of functions $\left( f_{n}\right) _{n\in 
\mathbb{N}
}$ where $f_{n}=\left( \vartheta _{1}\left( t_{n}\right) \right)
^{-1}T_{t_{n}}f$. This sequence is bounded in $A_{\vartheta _{1},\vartheta
_{2}}^{p,1,q,r}\left( 
\mathbb{R}
^{d}\right) $. Indeed, by Theorem \ref{statement}, we get 
\begin{eqnarray*}
\left\Vert f_{n}\right\Vert _{\vartheta _{1},\vartheta _{2}}^{p,1,q,r}
&=&\left( \vartheta _{1}\left( t_{n}\right) \right) ^{-1}\left\Vert
T_{t_{n}}f\right\Vert _{\vartheta _{1},\vartheta _{2}}^{p,1,q,r} \\
&\leq &\left( \vartheta _{1}\left( t_{n}\right) \right) ^{-1}\vartheta
_{1}\left( t_{n}\right) \left\Vert f\right\Vert _{\vartheta _{1},\vartheta
_{2}}^{p,1,q,r}=\left\Vert f\right\Vert _{\vartheta _{1},\vartheta
_{2}}^{p,1,q,r}
\end{eqnarray*}

Now, we will prove that there would not exists norm convergence of
subsequence of $\left( f_{n}\right) _{n\in 
\mathbb{N}
}$ in $\left( L_{\vartheta }^{p},\ell ^{1}\right) $. The sequence mentioned
above converges to zero in sense to the vague topology. To prove this, for
every $k\in C_{c}(%
\mathbb{R}
^{d})$ we obtain%
\begin{eqnarray}
\left\vert \dint\limits_{%
\mathbb{R}
^{d}}f_{n}(x)k(x)dx\right\vert &\leq &\frac{1}{\vartheta _{1}\left(
t_{n}\right) }\left\Vert k\right\Vert _{p^{\prime },\vartheta
_{1}}\left\Vert f_{n}\right\Vert _{p,\vartheta _{1}}  \notag \\
&\leq &\frac{C_{1}}{\vartheta _{1}\left( t_{n}\right) }\left\Vert
k\right\Vert _{p^{\prime },\vartheta _{1}}\left\Vert f_{n}\right\Vert
_{\vartheta _{1},\vartheta _{2}}^{p,1,q,r}  \label{tendtozero1}
\end{eqnarray}%
where $\frac{1}{p}+\frac{1}{p^{\prime }}=1$ and $C_{1}>0.$ Since the right
side of (\ref{tendtozero1}) tends zero for $n\longrightarrow \infty ,$ we
have 
\begin{equation*}
\dint\limits_{%
\mathbb{R}
^{d}}f_{n}(x)k(x)dx\longrightarrow 0.
\end{equation*}%
If we consider the Theorem \ref{vague}, the only possible limit of $\left(
f_{n}\right) _{n\in 
\mathbb{N}
}$ in $\left( L_{\vartheta }^{p},\ell ^{1}\right) $ is zero. This follows by
Lemma 2.2 in \cite{Feich3} that there exist $C_{2},C_{3}>0$ depending on $f$
such that%
\begin{equation*}
C_{2}\vartheta \left( y\right) \leq \left\Vert T_{y}f\right\Vert
_{p,\vartheta }\leq C_{3}\vartheta \left( y\right)
\end{equation*}%
for all $y\in 
\mathbb{R}
^{d}.$ Thus we have%
\begin{eqnarray}
\left\Vert f_{n}\right\Vert _{p1,\vartheta } &=&\left( \vartheta _{1}\left(
t_{n}\right) \right) ^{-1}\left\Vert T_{t_{n}}f\right\Vert _{p1,\vartheta
}\geq \frac{1}{C_{1}}\left( \vartheta _{1}\left( t_{n}\right) \right)
^{-1}\left\Vert T_{t_{n}}f\right\Vert _{p,\vartheta }  \notag \\
&\geq &\frac{C_{2}}{C_{1}}\left( \vartheta _{1}\left( t_{n}\right) \right)
^{-1}\vartheta \left( y\right) .  \label{nevercom}
\end{eqnarray}%
Since $\frac{\vartheta \left( t_{n}\right) }{\vartheta _{1}\left(
t_{n}\right) }\geq \delta >0$ for all $t_{n}$, by using (\ref{nevercom}) we
get%
\begin{equation*}
\left\Vert f_{n}\right\Vert _{p1,\vartheta }\geq C\left( \vartheta
_{1}\left( t_{n}\right) \right) ^{-1}\vartheta \left( y\right) \geq C\delta
\end{equation*}%
where $C=\frac{C_{2}}{C_{1}}.$ Thus there would not possible to find norm
convergent subsequence of $\left( f_{n}\right) _{n\in 
\mathbb{N}
}$ in $\left( L_{\vartheta }^{p},\ell ^{1}\right) $. This completes the
proof.
\end{proof}

\begin{theorem}
If $\vartheta _{3}\prec \vartheta _{1}$ and $\frac{\vartheta _{3}\left(
y\right) }{\vartheta _{1}\left( y\right) }$ doesn't tend to zero in $%
\mathbb{R}
^{d},$ then the embedding $i:$ $A_{\vartheta _{1},\vartheta
_{2}}^{p,1,q,r}\left( 
\mathbb{R}
^{d}\right) \hookrightarrow A_{\vartheta _{3},\vartheta
_{2}}^{p,1,q,r}\left( 
\mathbb{R}
^{d}\right) $ is never compact.
\end{theorem}

\begin{proof}
Since $\vartheta _{3}\prec \vartheta _{1}$, then it is clear that $%
A_{\vartheta _{1},\vartheta _{2}}^{p,1,q,r}\left( 
\mathbb{R}
^{d}\right) \subset A_{\vartheta _{3},\vartheta _{2}}^{p,1,q,r}\left( 
\mathbb{R}
^{d}\right) $. It is also known by Theorem \ref{closedgraph} that the unit
map from $A_{\vartheta _{1},\vartheta _{2}}^{p,1,q,r}\left( 
\mathbb{R}
^{d}\right) $ into $A_{\vartheta _{3},\vartheta _{2}}^{p,1,q,r}\left( 
\mathbb{R}
^{d}\right) $ is continuous. Now take any bounded sequence of $\left(
f_{n}\right) _{n\in 
\mathbb{N}
}$ in $A_{\vartheta _{1},\vartheta _{2}}^{p,1,q,r}\left( 
\mathbb{R}
^{d}\right) $. If there exists norm convergent subsequence of $\left(
f_{n}\right) _{n\in 
\mathbb{N}
}$ in $A_{\vartheta _{3},\vartheta _{2}}^{p,1,q,r}\left( 
\mathbb{R}
^{d}\right) $, this subsequence also converges in $\left( L_{\vartheta
_{3}}^{p},\ell ^{1}\right) $. But this is a contradiction because of the
fact that the embedding of the space $A_{\vartheta _{1},\vartheta
_{2}}^{p,1,q,r}\left( 
\mathbb{R}
^{d}\right) $ into $\left( L_{\vartheta _{3}}^{p},\ell ^{1}\right) $ is
never compact by the Theorem \ref{tendtozero}.
\end{proof}

Note that if we define the sequence in Theorem \ref{tendtozero} as $%
f_{n}=\left( \vartheta _{2}\left( t_{n}\right) \right) ^{-1}T_{t_{n}}f,$
then the proof of the following theorem is similar with the previous one.

\begin{theorem}
If $\vartheta _{1}\prec \vartheta _{2},\vartheta _{3}\prec \vartheta _{2}$
and $\frac{\vartheta _{3}\left( x\right) }{\vartheta _{2}\left( x\right) }$
does not tend to zero in $%
\mathbb{R}
^{d}$ as $x\longrightarrow \infty ,$ then the embedding of the space $%
A_{\vartheta _{1},\vartheta _{2}}^{p,1,q,r}\left( 
\mathbb{R}
^{d}\right) $ into $\left( L_{\vartheta _{3}}^{p},\ell ^{1}\right) $ is
never compact.
\end{theorem}

\begin{theorem}
The embedding of the space $A_{\vartheta _{1},\vartheta
_{2}}^{p,1,q,r}\left( 
\mathbb{R}
^{d}\right) $ into $A_{\vartheta _{3},\vartheta _{4}}^{p,1,q,r}\left( 
\mathbb{R}
^{d}\right) $ is never compact if

\begin{enumerate}
\item[\textit{(i)}] $\vartheta _{4}\prec \vartheta _{2}\prec \vartheta
_{1},\vartheta _{3}\prec \vartheta _{1}$ and $\frac{\vartheta _{3}\left(
x\right) }{\vartheta _{1}\left( x\right) }$ does not tend to zero in $%
\mathbb{R}
^{d}$ as $x\longrightarrow \infty ,$ or

\item[\textit{(ii)}] $\vartheta _{3}\prec \vartheta _{1}\prec \vartheta
_{2},\vartheta _{4}\prec \vartheta _{2}$ and $\frac{\vartheta _{3}\left(
x\right) }{\vartheta _{2}\left( x\right) }$ does not tend to zero in $%
\mathbb{R}
^{d}$ as $x\longrightarrow \infty .$
\end{enumerate}
\end{theorem}

\begin{proof}
Let $\vartheta _{4}\prec \vartheta _{2}\prec \vartheta _{1},\vartheta
_{3}\prec \vartheta _{1}.$ Thus, there exist $C_{1},C_{2}>0$ such that $%
\vartheta _{4}\left( x\right) \leq C_{1}\vartheta _{2}\left( x\right) $ and $%
\vartheta _{3}\left( x\right) \leq C_{2}\vartheta _{1}\left( x\right) $ for
all $x\in 
\mathbb{R}
^{d}.$ This follows that $A_{\vartheta _{1},\vartheta _{2}}^{p,1,q,r}\left( 
\mathbb{R}
^{d}\right) \subset A_{\vartheta _{3},\vartheta _{4}}^{p,1,q,r}\left( 
\mathbb{R}
^{d}\right) $ and the unit function $I:A_{\vartheta _{1},\vartheta
_{2}}^{p,1,q,r}\left( 
\mathbb{R}
^{d}\right) \longrightarrow A_{\vartheta _{3},\vartheta
_{4}}^{p,1,q,r}\left( 
\mathbb{R}
^{d}\right) $ is continuous. Now, suppose that $\frac{\vartheta _{3}\left(
x\right) }{\vartheta _{1}\left( x\right) }$ does not tend to zero in $%
\mathbb{R}
^{d}$ as $x\longrightarrow \infty $ and $\left( f_{n}\right) _{n\in 
\mathbb{N}
}$ is a bounded sequence in $A_{\vartheta _{1},\vartheta
_{2}}^{p,1,q,r}\left( 
\mathbb{R}
^{d}\right) .$ If any subsequence of $\left( f_{n}\right) _{n\in 
\mathbb{N}
}$ is convergent in $A_{\vartheta _{3},\vartheta _{4}}^{p,1,q,r}\left( 
\mathbb{R}
^{d}\right) ,$ then this subsequence is also convergent in $\left(
L_{\vartheta _{3}}^{p},\ell ^{1}\right) .$ However, this conflict with
Theorem \ref{tendtozero}, because of the fact that the embedding of $%
A_{\vartheta _{1},\vartheta _{2}}^{p,1,q,r}\left( 
\mathbb{R}
^{d}\right) $ into $\left( L_{\vartheta _{3}}^{p},\ell ^{1}\right) $ is
never compact. This completes \textit{(i)}. In similar way, \textit{(ii)}
can be proved.
\end{proof}

\section{\textbf{Multipliers of }$A_{\protect\vartheta _{1},\protect%
\vartheta _{2}}^{1,1,q,r}\left( G\right) $}

In this section, we discuss multipliers of the spaces $A_{\vartheta
_{1},\vartheta _{2}}^{1,1,q,r}\left( G\right) $. We define the space 
\begin{equation*}
M_{A_{\vartheta _{1},\vartheta _{2}}^{1,1,q,r}\left( G\right) }=\left\{ \mu
\in M(\vartheta _{1}):\left\Vert \mu \right\Vert _{M}\leq C(\mu )\right\}
\end{equation*}%
where%
\begin{equation*}
\left\Vert \mu \right\Vert _{M}=\sup \left\{ \frac{\left\Vert \mu \ast
f\right\Vert _{\vartheta _{1},\vartheta _{2}}^{1,1,q,r}}{\left\Vert
f\right\Vert _{11,\vartheta _{1}}}:f\in \left( L_{\vartheta _{1}}^{1},\ell
^{1}\right) ,\text{ }f\neq 0,\text{ }\widehat{f}\in C_{c}(\widehat{G}%
)\right\} .
\end{equation*}%
By \cite[Proposition 2.1]{Gur1}, we get $M_{A_{\vartheta _{1},\vartheta
_{2}}^{1,1,q,r}\left( G\right) }\neq \left\{ 0\right\} .$

\begin{theorem}
If $\vartheta _{1}$ satisfies (BD) condition, then for a linear operator $%
T:\left( L_{\vartheta _{1}}^{1},\ell ^{1}\right) \longrightarrow
A_{\vartheta _{1},\vartheta _{2}}^{1,1,q,r}\left( G\right) $ the assertions
below are equivalent:

\begin{enumerate}
\item[\textit{(i)}] $T\in Hom_{\left( L_{\vartheta _{1}}^{1},\ell
^{1}\right) }\left( \left( L_{\vartheta _{1}}^{1},\ell ^{1}\right)
,A_{\vartheta _{1},\vartheta _{2}}^{1,1,q,r}\left( G\right) \right) .$

\item[\textit{(ii)}] There exists a unique $\mu \in M_{_{A_{\vartheta
_{1},\vartheta _{2}}^{1,1,q,r}\left( G\right) }}$ such that $Tf=\mu \ast f$
for every $f\in \left( L_{\vartheta _{1}}^{1},\ell ^{1}\right) $. Moreover
the correspondence between $T$ and $\mu $ defines an isomorphism between $%
Hom_{\left( L_{\vartheta _{1}}^{1},\ell ^{1}\right) }\left( \left(
L_{\vartheta _{1}}^{1},\ell ^{1}\right) ,A_{\vartheta _{1},\vartheta
_{2}}^{1,1,q,r}\left( G\right) \right) $ and $M_{A_{\vartheta _{1},\vartheta
_{2}}^{1,1,q,r}\left( G\right) }$.
\end{enumerate}
\end{theorem}

\begin{proof}
It is known that $A_{\vartheta _{1},\vartheta _{2}}^{1,1,q,r}\left( G\right) 
$ is a $S_{\vartheta _{1}}$ space by Theorem \ref{segal}. Thus, we get the
desired result if we consider the Proposition 2.4 in \cite{Gur1}.
\end{proof}

\begin{theorem}
If $\vartheta _{1}$ satisfies (BD) condition and $T\in Hom_{\left(
L_{\vartheta _{1}}^{1},\ell ^{1}\right) }\left( A_{\vartheta _{1},\vartheta
_{2}}^{1,1,q,r}\left( G\right) \right) $, then there is a unique pseudo
measure $\sigma \in \left( A(\widehat{G})\right) ^{\ast }$ (see \cite{Reit})
such that $Tf=\sigma \ast f$ for all $f\in A_{\vartheta _{1},\vartheta
_{2}}^{1,1,q,r}\left( G\right) $.
\end{theorem}

\begin{proof}
It is known that $A_{\vartheta _{1},\vartheta _{2}}^{1,1,q,r}\left( G\right) 
$ is a $S_{\vartheta _{1}}$ space by Theorem \ref{segal} and a Banach module
over $\left( L_{\vartheta _{1}}^{1},\ell ^{1}\right) $ by Theorem \ref%
{module}. Thus, the proof is completed by Theorem 5 in \cite{Dogan}.
\end{proof}

\begin{theorem}
The multiplier space $Hom_{\left( L_{\vartheta _{1}}^{1},\ell ^{1}\right)
}\left( \left( L_{\vartheta _{1}}^{1},\ell ^{1}\right) ,\left( A_{\vartheta
_{1},\vartheta _{2}}^{1,1,q,r}\left( G\right) \right) ^{\ast }\right) $ is
isomorphic to $\left( L_{\vartheta _{1}^{-1}}^{\infty },\ell ^{\infty
}\right) \times \left( L_{\vartheta _{2}^{-1}}^{q^{\prime }},\ell
^{r^{\prime }}\right) /K$.
\end{theorem}

\begin{proof}
By Theorem \ref{module}, we write $\left( L_{\vartheta _{1}}^{1},\ell
^{1}\right) \ast A_{\vartheta _{1},\vartheta _{2}}^{1,1,q,r}\left( G\right)
=A_{\vartheta _{1},\vartheta _{2}}^{1,1,q,r}\left( G\right) $. Hence by
Corollary 2.13 in \cite{Rief} and Theorem \ref{isomorphic}, we have 
\begin{eqnarray*}
&&Hom_{\left( L_{\vartheta _{1}}^{1},\ell ^{1}\right) }\left( \left(
L_{\vartheta _{1}}^{1},\ell ^{1}\right) ,\left( A_{\vartheta _{1},\vartheta
_{2}}^{1,1,q,r}\left( G\right) \right) ^{\ast }\right) \\
&=&\left( \left( L_{\vartheta _{1}}^{1},\ell ^{1}\right) \ast A_{\vartheta
_{1},\vartheta _{2}}^{1,1,q,r}\left( G\right) \right) ^{\ast } \\
&=&\left( L_{\vartheta _{1}^{-1}}^{\infty },\ell ^{\infty }\right) \times
\left( L_{\vartheta _{2}^{-1}}^{q^{\prime }},\ell ^{r^{\prime }}\right) /K
\end{eqnarray*}%
where $\frac{1}{q}+\frac{1}{q^{\prime }}=1$ and $\frac{1}{r}+\frac{1}{%
r^{\prime }}=1.$
\end{proof}

\end{document}